\documentclass[a4paper,11pt]{amsart}
\usepackage[utf8]{inputenc}
\usepackage{amsmath,amssymb}
\usepackage[utf8]{inputenc}
\usepackage{amsmath}
\usepackage{amssymb}
\usepackage{amsthm}
\usepackage{mathrsfs}
\usepackage{enumerate}
\usepackage{tikz-cd}
\usetikzlibrary{decorations.pathreplacing,calligraphy}
\usepackage{todonotes}
\usepackage[foot]{amsaddr}
\usepackage{hyperref}
\setuptodonotes{color=red!40,fancyline}
\usepackage{color}
\usepackage{subcaption}
\usepackage{geometry}

\newtheorem{theorem}{Theorem}[section]
\newtheorem*{theorem*}{Theorem}
\theoremstyle{corollary}
\newtheorem{corollary}{Corollary}[theorem]
\newtheorem*{corollary*}{Corollary}
\theoremstyle{remark}
\newtheorem{remark}[theorem]{Remark}

\theoremstyle{corollary}
\newtheorem{lemma}[theorem]{Lemma}
\newtheorem{proposition}[theorem]{Proposition}

\numberwithin{equation}{section}

\title{Extension of derivations to forms}

\author{Manujith K. Michel}
\address{Indian Institute of Science Education and Research Mohali, Punjab 140306, India}
\email{manujithkmichel@gmail.com}

\author{Chitrarekha Sahu}
\email{chitrarekhasahu97@gmail.com}

\begin{document}

	\maketitle
\begin{abstract}
The problem of extending derivations of a field $F$ to an $F-$algebra $B$ is widely studied in commutative algebra and non-commutative ring theory. For example, every derivation of $F$ extends to $B$  if $B$ is a separable algebraic extension or a central simple algebra over $F.$ We unify and generalize these results by showing that a derivation $d$ of $F$ with the field of constants $C$ extends to a finite dimensional algebra $B$ if $B$ is a form of some $C-$algebra having a smooth automorphism scheme $\rm G$. Furthermore, we show that the set of derivations of $B$ that extend the derivation $d$ of $F$ is in bijection with the set of derivations $\delta$ such that $(Y,\delta)$ is a differential $\rm G_F-$torsor where $Y$ is the $\rm G_F-$torsor corresponding to $B$.
\end{abstract}

\section{Introduction}
	Let $R$ be a ring. An additive map $\delta: R\longrightarrow R$ is called a \emph{derivation} of $R$ if it satisfies the Leibniz rule, that is $\delta(ab)=\delta(a)b+a\delta(b)$ for all $a,b\in R$. An element $r\in R$ is said to be a $\emph{constant}$ of $R$ under the derivation $\delta$ if $\delta(r)=0$. In this article, we fix a field $F$ with a derivation $': F\longrightarrow F, x\mapsto x'$, and the set of constants of $F$ under this derivation is denoted by $C$. One can easily see that $C$ is a subfield of $F$. Let $\bar{F}$ denote the separable closure of $F$. From \cite{magid1994lectures}, the derivation $x\mapsto x'$ can be uniquely extended to $\bar F$ and is again denoted by $':\bar F\longrightarrow \bar F$. For a matrix $M=(m_{ij})\in M_n(\bar F)$, let $M'$ denote the matrix $(m'_{ij})$.\par
    
	 For a finite-dimensional $C-$algebra $A$, the automorphism scheme of $A$ is given by the functor ${\rm Aut}_{A}: S\mapsto {\rm Aut}_A(S):={\rm Aut}(A\otimes_C S/S)$ for every commutative $C-$algebra $S$, where ${\rm Aut}(A\otimes_C S/S)$ is the group of all $S-$linear automorphisms of $A\otimes _C S$. It is known that the functor ${\rm Aut}_{A}$ is an algebraic group scheme over $C$ \cite{serre1979galois}.  \par
    
	A \emph{form} of a  $C-$algebra $A$ over $F$ is an $F-$algebra $B$ such that $A\otimes_C \bar{F}\cong B\otimes_C \bar{F}$ as $\bar F-$algebras. It can be observed that the separable extensions and finite-dimensional central simple algebras over $F$ of degree $n$ are forms of $\prod^n C$ and ${\rm M_n}(C)$ respectively. Also the automorphism schemes of $\prod^n C$ and ${\rm M_n}(C)$ are smooth. It is well known that the derivations of $F$ can be extended to separable field extensions of $F$ and finite-dimensional central simple algebras over $F$\cite{amitsur1982extension}. In this paper, we will generalize these results by showing that the derivation of $F$ can be extended to any form of a given $C-$algebra $A$ over $F$ such that ${\rm Aut}_A$ is smooth. Let $A$ be a $C-$algebra with smooth automorphism scheme ${\rm G}:={\rm Aut}_A$. We fix a basis $\{e_1,\dots,e_n\}$ of $A$ over $C$ and identify ${\rm G}$ with a closed subscheme of ${\rm GL_n}$ and its Lie algebra $\mathfrak g$ with a subalgebra of ${\rm M_n}(C)$. The first Galois cohomology $H^1({\rm Gal}(\bar F/F), {\rm G}({\bar F}))$, where ${\rm Gal}(\bar F/F)$ is the Galois group of $\bar F$ over $F$ is in bijective correspondence with the set of isomorphism classes of forms of $A$ over $F$. Let $f\in H^1({\rm Gal}\bar F/F), {\rm G}({\bar F}))$ then there exists $P_f\in {\rm GL_n}(\bar F)$ such that $$f:{\rm Gal}(\bar F/F)\rightarrow {\rm G}(\bar F)
    $$ is given by, 
    \begin{equation}
     \sigma \mapsto f_{\sigma}:= P^{-1}_f\sigma(P_f) \label{fmap}
    \end{equation}
    
    for every $\sigma\in {\rm Gal}(\bar F/ F)$. We will show that the set of derivations of the corresponding form $A_f$ extending the derivation of $F$ can be identified with the set $[P^{-1}_f\mathfrak g_{\bar F}P_f+(P_f^{-1})'P_f]\bigcap {\rm M_n}(F)$, where $\mathfrak g_{\bar F}$ is the Lie algebra of the group scheme ${\rm G}_{\bar F}$ obtained by base change. For the above $P_f$ the set $[P^{-1}_f\mathfrak g_{\bar F}P_f+(P_f^{-1})'P_f]\bigcap {\rm M_n}(F)$ is non-empty \cite{juan2007equivariant} and therefore there exists a derivation of the form $A_f$ extending the derivation of $F$.   \par
	
 A \emph{${\rm G}_F-$torsor} is an affine scheme $Y$ of finite type over $F$ with a morphism $\rho:{\rm G}_F \times_F Y\longrightarrow Y$ of schemes such that for every $F$-algebra $R$, $\rho_R:G_F(R)\times_F Y(R)\longrightarrow Y(R)$ defines a left group action of ${\rm G}_F(R)$ on $Y(R)$ and the morphism $(\rho, p_2):{\rm G}_F\times_F Y\longrightarrow Y\times_F Y$ is an isomorphism of schemes, where $p_2$ is the projection on $Y$. If in addition, there is a derivation $\delta$ of $F[Y]$ that extends the derivation of $F$ such that the diagram

\[ \begin{tikzcd}
F[Y] \arrow{r}{\rho_*} \arrow[swap]{d}{\delta} & C[{\rm G}]\otimes_F F[Y] \arrow{d}{{\rm id}\otimes\delta} \\%
F[Y] \arrow{r}{\rho_*}&  C[{\rm G}]\otimes_F F[Y]
\end{tikzcd}
\]
 
commutes then the pair $(Y,\delta)$ is called a \emph{differential $G_F$-torsor}. The map $\rho_*:F[Y]\longrightarrow F[G]\otimes_F F[Y]\cong C[{\rm G}]\otimes_C F[Y]$ is the homomorphism induced by $\rho$ on the coordinate rings. We see that
for a ${\rm G}_F-$torsor $Y$, $(Y,\delta)$ is differential ${\rm G}_F-$torsor if and only if the derivation $$h\otimes x\mapsto \delta(h)\otimes x+h\otimes x', \quad h\in F[Y], \  x\in \bar F$$ of $F[Y]\otimes_F \bar F$ commutes with the action of ${\rm G}(\bar F)$ on $F[Y]\otimes_F \bar F$. More details on differential torsors can be found in \cite{bachmayr2018differential}. The set $H^1({\rm Gal}(\bar F/F), {\rm G}({\bar F}))$ is in bijection with the set of isomorphism classes of $G_F$-torsors over $F$. From \cite[Section 5]{juan2007equivariant},  we can see that, for a fixed cocycle $f\in H^1({\rm Gal}(\bar F/F), {\rm G}({\bar F}))$, if $Y_f$ is the corresponding ${\rm G}_F$-torsor then the set of derivations $\{\delta \ :\ (Y_f,\delta) \ $is a differential \ ${\rm G}_F-$torsor$\}$ is characterized by set $[P^{-1}_f\mathfrak g_{\bar F}P_f+(P_f^{-1})'P_f]\bigcap {\rm M_n}(F)$, where $P_f$ is same as in Equation (\ref{fmap}). So from all the above findings, we conclude this note by saying that the derivations of a form $A_f$ of $A$ over $F$ extending the derivation of $F$ are classified by the derivations $\delta$ of $F[Y_f]$ such that $(Y_f, \delta)$ is a differential ${\rm G}_F-$torsor. This last conclusion is a generalization of Theorem 5.1 in\cite{TSUI202449} which is proved using the Picard-Vessiot theory under the assumption that $C$ is algebraically closed.

\section*{Acknowledgements}
The authors thank their supervisor, Dr. Varadharaj R Srinivasan, for insightful discussions. The first author is funded by the Council of Scientific and Industrial Research (CSIR), and the second author is funded by the University Grants Commission (UGC).
	 
	 \section{Extension of derivation}
Let $C$ be a field and $A$ be a $C-$algebra with smooth automorphism scheme $\rm G$. Recall that for any $C-$algebra $S,{\rm G}(S):={\rm Aut}(A\otimes_C S/S)$. Now we shall observe that the $C-$linear derivations of $A$ can be identified with the Lie algebra of the group scheme ${\rm G}$. Let $A[x]$ be a polynomial ring in $x$ over $A$ such that $x$ commutes with the elements of $A$ and $\epsilon$ be the image of $x$ in $\frac{A[x]}{\langle x^2\rangle}$. We know that a map $\delta:A\longrightarrow A$ is a $C-$linear derivation of $A$ if and only if the map
\begin{align*}
A\longrightarrow A[\epsilon]  \\ 
a\mapsto a+\delta(a)\epsilon & \quad \text{for all} \ a\in A,
\end{align*}
is a $C-$algebra homomorphism\cite{magid1994lectures}. Let $\eta: A[\epsilon]\longrightarrow A$ be the $C-$linear map defined by $\eta(a+b\epsilon)=a$ for every $a,b\in A.$ If $\mathfrak g$ denotes the lie algebra of $G$ then using \cite[Section 12.2]{waterhouse2012introduction}, $$ \mathfrak g=\{\phi\in {\rm G}(C[\epsilon])\ : \ \phi_C=e_C\in {\rm G}(C)\},$$ where $e_C$ is the identity element of ${\rm G}(C)={\rm Aut}(A/C)$ and for every $\phi\in {\rm G}(C[\epsilon]), \phi_C$ is the composition map, 
	\[ A\hookrightarrow A[\epsilon]\xrightarrow{\phi} A[\epsilon]\xrightarrow{\eta} A .\]
Let $\phi \in  \mathfrak g, a\in A$ and $\phi(a)=a_0+a_1\epsilon$ for some $a_0,a_1\in A$ then $a=\phi_C(a)=a_0$. Therefore, $\phi(a+b\epsilon)=\phi(a)+\phi(b)\epsilon=a+a_1\epsilon+b\epsilon$. It can be easily checked that the map $\delta_{\phi}:A\rightarrow A$ defined by $\delta_{\phi}(a)=a_1$ is a $C-$linear derivation of $A$. 
	
	\begin{lemma} \label{phitodeltaphi}
		The map $\phi\mapsto \delta_{\phi}$ is a bijection from $\mathfrak g$ to the set of all $C-$linear derivations of $A$.
	\end{lemma}
\begin{proof}
   Let $\phi \in  \mathfrak g$ then $\phi$ is a $C-$algebra automorphism of $A[\epsilon]$. Let $ a\in A$ then $\phi(a)=a_0+a_1\epsilon$ for some $a_0,a_1\in A$. Since $a=\phi_C(a)=\eta\circ \phi(a)$ we get $a=a_0.$ This implies that the map $a\mapsto a+a_1\epsilon$ gives a $C-$algebra homomorphism from $A$ to $A[\epsilon]$ which gives a dervation $\delta_{\phi}:a\mapsto a_1$ of $A$. Conversely, if $\delta$ is a derivation of $A$ then the $C[\epsilon]-$linear extension $\phi_{\delta}$ of the map $a\mapsto a+\delta(a)\epsilon$ is a $C-$algebra automorphism of  $A[\epsilon]$ such that the restriction of $\eta\circ \phi_{\delta}$ to $A$ is an identity map that is $\phi_{\delta}\in \mathfrak g$. Clearly, the maps $\phi\mapsto \delta_{\phi}$ and $\delta \mapsto \phi_{\delta}$ are inverse of each other. 
\end{proof}	
	From the above Lemma (\ref{phitodeltaphi}), we can identify the $C-$linear derivations of $A$ with $\mathfrak g$. Let $F$ be a differential field with the set of constants $C$ and $\delta_0$ be the derivation of $A\otimes_C\bar F$ defined by $\delta_0(a\otimes f)=a\otimes x'$ for every $a\in A$ and $x\in \bar F$.

Let $\{e_1,\dots, e_n\}$ be a basis of $A$ over $C$, $f \in H^1({\rm Gal}(\bar F/F), {\rm G}({\bar F}))$ and $P_f\in {\rm GL_n}(\bar F)$ be as in Equation (\ref{fmap}). The form $A_f$ corresponding to $f$ is given by the invariants of the twisted action $$*:Gal(\bar F/F)\times (A\otimes_C \bar F)\rightarrow A\otimes_C\bar F $$ $$ (\sigma, a\otimes x)\mapsto f_{\sigma}(a\otimes\sigma x)$$ for every $\sigma\in Gal(\bar F/F), a\in A$ and $x \in \bar F$. Denote the transpose of $[e_1\otimes 1,\dots, e_n\otimes 1]$ by $V$ then the entries of $P^{-1}_fV$ are invariant under the twisted action and form a basis of $A_f$.

\begin{remark}
Observe that from Lemma (\ref{phitodeltaphi}) we can identify $\mathfrak g_{\bar F}$ with the set of all $ M\in {\rm M_n}(\bar F)$ such that the $\bar F-$linear map given by $\delta_M (V)=MV$ is a derivation of $A\otimes_C \bar F.$
\end{remark}

\begin{proposition}\label{mainresult}
	Let $f\in H^1({\rm Gal}(\bar F/F), {\rm G}({\bar F}))$ then there is a bijection from the set of all derivations of the corresponding form $A_f$ extending the derivation of $F$ to the set $$\{P^{-1}_fMP_f+(P^{-1}_f)'P_f : M\in \mathfrak g_{\bar F}\} \bigcap {\rm M_n}(F).$$
\end{proposition}
\begin{proof}
We will show that a map $\delta:A_f\longrightarrow A_f$ extending the derivation of $F$ given by $\delta(P^{-1}_fV)=N(P^{-1}_fV)$ where $N\in {\rm M_n}(F)$ is a derivation of $A_f$ if and only if $N\in \{P^{-1}_fMP_f+(P^{-1}_f)'P_f : M\in \mathfrak g_{\bar F}\} \bigcap {\rm M_n}(F).$ Let $N=P^{-1}_fMP_f+(P^{-1}_f)'P_f\in {\rm M_n}(F)$ for some $M \in \mathfrak g_{\bar F}$ then $\delta_M(V)= MV$ gives an $\bar F-$ linear derivation of $A\otimes_C \bar F$ and  $\delta:=\delta_0+\delta_M$ is a derivation of $A\otimes_C \bar F$ extending the derivation of $F$. 
    
We have, $$\delta(P^{-1}_fV)=(\delta_0+\delta_M)(P^{-1}_fV)=(P^{-1}_f)'V+P^{-1}_fMV=(P^{-1}_fMP_f+(P^{-1}_f)'P_f)(P^{-1}_fV)=N(P^{-1}_fV).$$
Since $A_f$ is the $F-$span of the entries of $P^{-1}_fV$ and  $N\in {\rm M_n}(F), \delta$ restricts to a derivation of $A_f$.\\
Conversely, if $\delta$ is a derivation of $A_f$ extending the derivation of $F$ then $\delta(P^{-1}_fV)=N(P^{-1}_fV)$ for some $N\in {\rm M_n}(F)$. The $C-$linear map $$b\otimes x\mapsto \delta(b)\otimes x+b\otimes x', \quad b\in A_f, \ x\in \bar F,$$ is a derivation of $A_f\otimes_F\bar F$ extending $\delta$ to $A_f\otimes_F \bar F$. Since $A_f\otimes_F \bar F \cong A\otimes_C \bar F$, we get a derivation of $A\otimes_C \bar F$. We will denote this derivation on $A\otimes \bar F$ by the same $\delta$. Let $M\in {\rm M_n}(\bar F)$ such that $\delta(V)=MV$. Then $(\delta-\delta_0)(V)=MV$ is  an  $\bar F-$linear derivation of $A\otimes_C \bar F$ and therefore $M \in \mathfrak g_{\bar F}$. Also, $$\delta(P^{-1}_fV)=(P^{-1}_f)'V+P^{-1}_fMV=((P^{-1}_f)'P_f+P^{-1}_fMP_f)(P^{-1}_fV).$$ Since $\delta(P^{-1}_fV)=N(P^{-1}_fV)$ and entries of $P^{-1}_fV$ form a basis of $A_f$ over $F$ we obtain $N=(P^{-1}_f)'P_f+P^{-1}_fMP_f.$
	 \end{proof}

\begin{corollary}
Let $B$ be a form of $A$ over $F$ then there is a derivation of $B$ extending the derivation of $F$.
\end{corollary}
\begin{proof}
We know that $B$ is of the form $B=A_f$ for some $f\in H^1(Gal(\bar F/F), {\rm G}({\bar F}))$. Since the set $\{(P^{-1}_f)'P_f+P^{-1}_fMP_f):M\in\mathfrak g_{\bar F}\}\bigcap {\rm M_n}(F)$ is non-empty\cite[Section 5]{juan2007equivariant},  the result follows from the Proposition (\ref{mainresult}). 
\end{proof}

Let $f\in H^1(Gal(\bar F/F), {\rm G}({\bar F}))$ and $Y_f$ be the ${\rm G}_F-$torsor corresponding to $f$. Recall that the set of derivations $\{\delta \ :\ (Y_f,\delta) \ $is a differential \ ${\rm G}_F-$torsor$\}$ is characterized by the set $[P^{-1}_f\mathfrak g_{\bar F}P_f+(P_f^{-1})'P_f]\bigcap {\rm M_n}(F)$. Along with Proposition (\ref{mainresult}) we will have the following bijection.

\begin{corollary}
Let $f\in H^1(Gal(\bar F/F), {\rm G}({\bar F}))$, $A_f$ be the corresponding form and $Y_f$ be the corresponding ${\rm G}_F-$torsor then the set of all the derivations of $A_f$ extending the derivation of $F$ is in bijection with the set $\{\delta \ :\ (Y_f,\delta) \ $is a differential \ ${\rm G}_F-$torsor$\}$.
\end{corollary}

\bibliographystyle{alpha}
	\bibliography{CM}

\begin{thebibliography}{BHHW18}

\bibitem[Ami82]{amitsur1982extension}
SA~Amitsur.
\newblock Extension of derivations to central simple algebras.
\newblock {\em Communications in Algebra}, 10(8):797--803, 1982.

\bibitem[BHHW18]{bachmayr2018differential}
Annette Bachmayr, David Harbater, Julia Hartmann, and Michael Wibmer.
\newblock Differential embedding problems over complex function fields.
\newblock {\em Documenta Mathematica}, 23:241--291, 2018.

\bibitem[JL07]{juan2007equivariant}
Lourdes Juan and Arne Ledet.
\newblock Equivariant vector fields on non-trivial son-torsors and differential galois theory.
\newblock {\em Journal of Algebra}, 316(2):735--745, 2007.

\bibitem[Mag94]{magid1994lectures}
Andy~R Magid.
\newblock {\em Lectures on differential Galois theory}, volume~7.
\newblock American Mathematical Society Providence, RI, 1994.

\bibitem[Ser79]{serre1979galois}
Jean-Pierre Serre.
\newblock {\em Galois cohomology}.
\newblock Springer, 1979.

\bibitem[TW24]{TSUI202449}
Man~Cheung Tsui and Yidi Wang.
\newblock Cohomology for picard-vessiot theory.
\newblock {\em Journal of Algebra}, 658:49--72, 2024.

\bibitem[Wat12]{waterhouse2012introduction}
William~C Waterhouse.
\newblock {\em Introduction to affine group schemes}, volume~66.
\newblock Springer Science \& Business Media, 2012.

\end{thebibliography}
\end{document}